\documentclass[11pt]{article}
\usepackage[noadjust]{cite}
\usepackage{amsthm, amsmath, amssymb, amsfonts, float, url, booktabs, tikz, setspace, fancyhdr, enumerate}
\usepackage{ booktabs, array, multirow, xcolor}
\usepackage{multirow} 
\usepackage{makecell}
\usepackage{enumitem}
\usepackage[margin = 1in]{geometry}
\usepackage{pgfplots}
\pgfplotsset{compat=1.18}
\usepackage[czech,english]{babel} 
\usepackage[hidelinks]{hyperref}
\usepackage[capitalise]{cleveref}
\usepackage{tikz}
\usepackage{comment}
\usepackage{xcolor}
\usepackage{cite}

\usepackage[textsize=scriptsize,colorinlistoftodos]{todonotes}

\newtheorem{theorem}{Theorem}[section]
\newtheorem{proposition}[theorem]{Proposition}
\newtheorem{lemma}[theorem]{Lemma}
\newtheorem{corollary}[theorem]{Corollary}

\newtheorem{problem}[theorem]{Problem}

\theoremstyle{definition}
\newtheorem{definition}[theorem]{Definition}

\theoremstyle{remark}

\let\le\leqslant
\let\ge\geqslant

\begin{document}

\title{Two problems of Burr, Erd\H os, Graham, and S\'os on maximal anti-Ramsey functions for $P_4$}
\author{Mingze Li\thanks{School of Mathematical Sciences, University of Science and Technology of China, Hefei, Anhui 230026, China. Email: \url{lmz10@mail.ustc.edu.cn}.} 
\and Bo Ning\thanks{College of Computer Science, Nankai University, Tianjin 300350, P.R. China. Research supported by National Natural Science Foundation of China (grant No. 12371350) and the Fundamental Research Funds for the Central Universities, Nankai University (No. 63263259). Email: \url{bo.ning@nankai.edu.cn}.} \and Tianying Xie\thanks{School of Mathematics and Statistics, Fuzhou University, Fuzhou, Fujian 350108, China. Research supported by National Key Research and
Development Program of China 2023YFA1010201 and National Natural Science Foundation of China grants 12501474 and 12471336. Email: \url{xiety@fzu.edu.cn}.}}
\date{June 25, 2026}
\maketitle

\begin{abstract}
Burr, Erd\H os, Graham, and S\'os introduced the maximal anti-Ramsey function $\chi_{\mathrm{S}}(n,e,L)$, the minimum number of colors required over all $n$-vertex graphs with at least $e$ edges such that every copy of $L$ is rainbow. 
In \cite{BEGS1989}, they posed the following two problems: (i) Is it true that there exists $C>0$, such that for all $u\ge 1$, $\chi_{\mathrm{S}}\left(n,\lfloor un \rfloor,P_4 \right)<Cu$ holds for all sufficiently large $n$?  (ii) Is it true that for all $\epsilon >0$, there exists $c(\epsilon)>0$ such that for all sufficiently large $n$, \\ $\chi_{\mathrm{S}}\left(n,\binom{n}{2}-\lfloor n^{2-\epsilon} \rfloor,P_4 \right)>c(\epsilon)n^{2}$? 
In this note, we give an affirmative answer to the first problem and a negative answer to the second problem.

For the first problem, our proof uses a local density inequality with strong edge-colorings of odd Kneser graphs.
In particular, our proof uses the characterization by Lu\v{z}ar, M\'{a}\v{c}ajov\'a, \v{S}koviera, and Sot\'ak of~$k$-regular graphs whose strong chromatic index equals~$2k-1$.
For the second result, our main tool is the construction of Alon, Moitra, and Sudakov. We show that for every fixed~$0<\epsilon<1/2$ there exist~$\gamma>0$ and arbitrarily large~$n$ such that~$\chi_{\mathrm{S}}\bigl(n,\tbinom{n}{2}-\lfloor n^{2-\epsilon}\rfloor,P_4\bigr)\;\le\; n^{2-\gamma}=o(n^{2}).$
\end{abstract}

\section{Introduction}

Ramsey theory asserts, in many different forms, that sufficiently large structures contain highly organized substructures. In the graph setting, the ambient structure is usually an edge-colored graph, and the organized subgraph is monochromatic. Anti-Ramsey theory, initiated by Erd\H os, Simonovits, and S\'os~\cite{ESS1975}, asks for the opposite extreme: a subgraph is called \emph{rainbow} if all its edges have distinct colors, and one seeks conditions forcing such rainbow subgraphs.

Burr, Erd\H os, Graham and S\'os \cite{BEGS1989} introduced a dual version of this problem. For graphs $G$ and $L$, let $\chi_{\mathrm{S}}(G,L)$ denote the minimum number of colors in an edge-coloring of $G$ in which every copy of $L$ is rainbow. 
\begin{definition}
    For a graph $L$ and positive integers $n,e$, we define $\chi_{\mathrm S}(n,e,L)$ as the minimum number of colors needed to color the edges of an $n$-vertex graph $G$ with at least $e$ edges such that every copy of $L$ in $G$ is rainbow. 
\end{definition}
Equivalently, it is the minimum over all such graphs:
    \[\chi_{\mathrm{S}}(n,e,L) = \min \{\chi_{\mathrm{S}}(G,L):|V(G)| = n,|E(G)| = e\}.\] 
The minimization over $G$ is essential. 
In other words, $\chi_{\mathrm{S}}(n,e,L)-1$ is the largest number $r$ such that for any graph $G$ with $n$ vertices and $e$ edges, every edge-coloring of $G$ with $r$ colors must contain a non-rainbow copy of $L$. 
This function is closely related to the Tur\'an number. For a graph $L$, the Tur\'{a}n number of $L$ is denoted by $\mathrm{ex}(n,L)$, which represents the maximum number of edges in an $n$-vertex graph that avoids copies of $L$. 
If $e \le \mathrm{ex}(n,L)$, then a subgraph of an extremal $L$-free graph shows $\chi_{\mathrm{S}}(n,e,L) = 1$. The problem becomes meaningful only beyond the Tur\'an threshold, where copies of $L$ are forced but may overlap in ways that allow economical colorings. 

A well-known result~\cite{E1962} in extremal graph theory implies that $\mathrm{ex}(n,C_{2k+1})=\lfloor n^{2}/4 \rfloor$ holds for all $k\ge 1$ when $n$ is sufficiently large; see also the complete characterization of extremal graphs for odd cycles by F\"uredi and Gunderson~\cite{FG2014}. Motivated by this threshold, Burr, Erd\H{o}s, Graham, and S\'os~\cite{BEGS1989} studied $\chi_{\mathrm{S}}\left(n,\lfloor n^{2}/4 \rfloor+1,C_{2k+1}\right)$. They obtained results for small odd cycles and conjectured that, for every $k\ge3$, this quantity is $n^2/8+o(n^2)$. This conjecture was later reiterated by Erd\H{o}s~\cite{E1991}.
Recently, Buci\'c, Chen, and Ma~\cite{BCM2026} proved that for an integer $k\geq 4$ and $\frac{n^2}{4}+1\leq e\leq \binom{n}{2}$, $\chi_{\mathrm{S}}\left(n,e,C_{2k+1} \right)=\frac{e}{2}+\frac{n}{2}\sqrt{e-\frac{n^2}{4}}+o(n^2),$ confirming the conjecture for all $k\ge 4$ in a stronger form. 
Thus, the case
$k=3$ remains the only
open case.

In addition to odd cycles, several results have been obtained when $L$ is a bipartite graph. Burr, Erd\H os, Frankl, Graham, and S\'os~\cite{BEFGS1991} have investigated both cases regarding whether $L$ contains two strongly independent edges. Burr, Erd\H os, Graham, and S\'os~\cite{BEGS1989} determined the values of $\chi_{\mathrm{S}}\left(n,\lfloor un \rfloor,P_{k} \right)$ for $k=3$, and for the range $u\ge k\ge 5$. Later, S\'ark\"ozy and Selkow~\cite{SS2006} obtained a linear lower bound for connected bipartite graphs other than complete bipartite graphs. 

In this paper, we focus on the case $L=P_4$, the path on four vertices, which is one of the central small bipartite cases in the work of Burr, Erd\H os, Graham, and S\'os~\cite{BEGS1989}. It is closely related to Ruzsa--Szemer\'edi graphs, since an edge-coloring whose color classes are induced matchings automatically guarantees that every $P_4$ is rainbow. Fox, Huang, and Sudakov~\cite{FHS2017} studied Ruzsa-Szemer\'edi graphs whose induced matchings have linear size.

Burr, Erd\H{o}s, Graham, and S\'os~\cite{BEGS1989} posed the following two open problems for $P_4$.
\begin{problem}[Burr, Erd\H os, Graham, and S\'os~\cite{BEGS1989}]\label{problem}
\ 
 Is it true that there exists $C>0$, such that for all $u\ge 1$, $\chi_{\mathrm{S}}\left(n,\lfloor un \rfloor,P_4 \right)<Cu$ holds for all sufficiently large $n$?
 \end{problem}

\begin{problem}[Burr, Erd\H os, Graham, and S\'os~\cite{BEGS1989}]\label{problem-2}
\
 Is it true that for all $\epsilon >0$, there exists $c(\epsilon)>0$ such that for all sufficiently large $n$, $\chi_{\mathrm{S}}\left(n,\binom{n}{2}-\lfloor n^{2-\epsilon} \rfloor,P_4 \right)>c(\epsilon)n^{2}$?  
 \end{problem}
   
Our first result provides a positive answer to Problem~\ref{problem}. 

\begin{theorem}
\label{thm:coarse}
For every fixed $u\ge 1$, there exists $n_0=n_0(u)>0$ such that for $n>n_0$, we have $4u - 7 - o(1)\le \chi_{\mathrm{S}}(n,\lfloor un\rfloor ,P_{4})\le 4\lfloor u\rfloor + 3$. 
Moreover, if $u > 2$, then $4u-1-o(1)\le \chi_{\mathrm{S}}(n,\lfloor un\rfloor ,P_{4})$. 
Here $o(1) \to 0$ as $n \to \infty$, with $u$ fixed.
\end{theorem}

The sharper statement is governed by the following density thresholds. For $q \in \mathbb{N}$, define
\[
f(q) = \begin{cases}
\frac{d}{2} & \text{if } q = 2d-1,\\[4pt]
\frac{d(d+1)}{2d+1} & \text{if } q = 2d.
\end{cases}
\]

The values $f(2d-1)=d/2$ and $f(2d)=d(d+1)/(2d+1)$ are attained by $\mathrm{KG}(2d-1,d-1)$ and by deleting one color class from the canonical strong edge-coloring of $\mathrm{KG}(2d+1,d)$, respectively.

\begin{theorem}
\label{thm:fixed}
Let $u>0$ be fixed. Define $q_0(u)=\min\{q\in\mathbb N:f(q)\ge u\}$ and $q_+(u)=\min\{q\in\mathbb N:f(q)>u\}$. The following statements hold.

\begin{itemize}
    \item[(i)]
    If $0<u<1$, then $\chi_{\mathrm S}(n,\lfloor un\rfloor,P_4)=1$ for all sufficiently large $n$.

    \item[(ii)]
    If $u=1$, then, for every $n\ge3$,
    \[
    \chi_{\mathrm S}(n,n,P_4)=
    \begin{cases}
        1, & 3\mid n,\\
        3, & 3\nmid n.
    \end{cases}
    \]

    \item[(iii)]
    If $1<u\le 2$, then $3\le\chi_{\mathrm S}(n,\lfloor un\rfloor,P_4)\le q_+(u)$ for all sufficiently large $n$.

    \item[(iv)]
    If $u>2$, then $q_0(u)\le\chi_{\mathrm S}(n,\lfloor un\rfloor,P_4)\le q_+(u)$ for all sufficiently large $n$. In particular, if $2u\notin \mathbb{Z}$ and $u\neq f(2\lfloor 2u\rfloor)$, then $\chi_{\mathrm S}(n,\lfloor un\rfloor,P_4)=q_+(u)$ for all sufficiently large $n$.

    \item[(v)]
    If $u>2$ is an integer and $N_u=\binom{4u-1}{2u-1}$, then, for all sufficiently large $n$,
    \[
    \chi_{\mathrm S}(n,un,P_4)=
    \begin{cases}
        4u-1, & N_u\mid n,\\
        4u, & N_u\nmid n.
    \end{cases}
    \]
\end{itemize}
\end{theorem}

Our second result concerns Problem~\ref{problem-2}, which asks whether deleting only $n^{2-\epsilon}$ edges from $K_n$ still forces quadratically many colors for every fixed $\epsilon>0$. The answer is negative for $0<\epsilon < 1/2$.

\begin{theorem}
\label{thm:near}
For every fixed $0<\epsilon < 1/2$, there exists $\gamma = \gamma(\epsilon) > 0$ such that for arbitrarily large $n$, $\chi_{\mathrm{S}}\left(n,\binom{n}{2} - \lfloor n^{2-\epsilon}\rfloor ,P_{4}\right) \le n^{2-\gamma} = o(n^{2})$. 
Moreover, one may take $\gamma = \Omega((1/2 - \epsilon)^{3})$.
\end{theorem}

\noindent
{\bf \underline{Proof Sketch.}}
In the proofs of Theorems~\ref{thm:coarse} and~\ref{thm:fixed}, the proofs proceed in two directions. In order to show the lower bound of $\chi_{\mathrm{S}}\left(n,\lfloor un \rfloor ,P_{4}\right)$, we observe that if $G$ is a $q$-edge-colored graph in which every $P_4$ is rainbow, then for any edge $xy\in E(G)$, all edges incident with $x$ or $y$ are assigned distinct colors. With this observation, we obtain a density inequality for these $q$-edge-colored graphs (see Lemma~\ref{lem:density}), which will give the lower bound. 
For the upper bound, we construct graphs with prescribed density and strong chromatic index. Starting from a given graph $F$, we take several disjoint copies and delete edges, which are operations that do not increase the density or the strong chromatic index. Applying this process to the graph $F_q$, which attains the bound in Lemma~\ref{lem:density}, yields the desired graph. When $u>2$ is an integer and $\binom{4u-1}{2u-1}\nmid n$, the density inequality forces any graph attaining $4u-1$ colors to be $2u$-regular with strong chromatic index $4u-1$ (see Theorem~\ref{thm:Luzar}), the characterization of Lu\v{z}ar, M\'{a}\v{c}ajov\'a, \v{S}koviera, and Sot\'ak~\cite{LMSS2022} for $k$-regular graphs with strong chromatic index $2k-1$, then implies that such a graph must cover the odd Kneser graph $\mathrm{KG}(4u-1,2u-1)$, which by Lemma~\ref{fiber} forces $\binom{4u-1}{2u-1}\mid n$, a contradiction. Hence $\chi_{\mathrm{S}}(n,un,P_4)=4u$ in this case.

To prove Theorem~\ref{thm:near}, we use a construction of Alon, Moitra, and Sudakov~\cite{AMS2013}. The constructed graph $H_n$ has $n^{2-\epsilon}$ fewer edges than the complete graph $K_n$ (on the same vertex set) and can be covered by $o(n^2)$ induced matchings. This yields the desired conclusion. 

\noindent
{\bf \underline{Organization}}
The paper is organized as follows. Section \ref{sec2} gives all preliminaries for the proofs. In Section \ref{sec3}, we prove Theorems \ref{thm:coarse}, \ref{thm:fixed}, and \ref{thm:near}. 

\section{Preliminaries}\label{sec2}

All graphs are finite and simple, and copies of a graph are not required to be induced. For a graph $G$, let $V(G)$, $E(G)$, $N_{G}(v)$ and $d_G(v)$ denote its vertex set, edge set, the neighborhood of $v$ and the degree of $v$. When the graph is clear, we write $N(v)$ and $d(v)$. An edge set $M \subseteq E(G)$ is an \emph{induced matching} if the subgraph induced by the endpoints of edges in $M$ is exactly the matching $M$. A \emph{strong edge-coloring} is an edge-coloring whose color classes are induced matchings. For a graph $G$, the \emph{strong chromatic index} of $G$, denoted by $\chi^\prime_s(G)$, is the least number of colors in a strong edge-coloring of $G$. 

The following lemma shows a simple relationship between the strong chromatic index and the function $\chi_{\mathrm{S}}(G,P_4)$. 
\begin{lemma}
\label{lem:cover} 
If the edge set of a graph $G$ can be covered by at most $q$ induced matchings, then $\chi_{\mathrm{S}}(G,P_4)\le q$. In particular, $\chi_{\mathrm{S}}(G,P_4) \le \chi^\prime_s(G)$.
\end{lemma}
\begin{proof}
Assign to every edge one of the induced matchings containing it, and color the resulting subsets with the corresponding colors. Subsets of induced matchings are induced matchings. In a copy of $P_4$, any two edges are either adjacent or are the two end-edges joined by the middle edge. Hence two edges of one $P_4$ cannot lie in the same induced matching, and every $P_4$ is rainbow.
\end{proof}

\subsection{Kneser graph and odd graph}
We use Kneser graphs in the standard notation: 
\begin{definition}
    For positive integers $m\ge r$, the \emph{Kneser graph} $\mathrm{KG}(m,r)$ is defined as follows: 
    \begin{align*}
        V(\mathrm{KG}(m,r))&= \binom{[m]}{r}, \\
        E(\mathrm{KG}(m,r))&= \left\{ AB: A, B\in \binom{[m]}{r} \mbox{ and } A\cap B=\emptyset \right\}. 
    \end{align*}
\end{definition}
\begin{definition}
    For $k \ge 2$, we define the \emph{odd graph} $\mathrm{O}_k=\mathrm{KG}(2k-1,k-1)$. 
\end{definition}
The following property of the odd graph $\mathrm{O}_k$ follows directly from the definition.
\begin{proposition}
    \label{prop:Kneser}
    For $k \ge 2$, the odd graph $\mathrm{O}_k$ is a $k$-regular connected graph with $|V(\mathrm{O}_k)|=\binom{2k-1}{k-1}$. 
\end{proposition}

Faudree, Schelp, Gy\'arf\'as, and Tuza~\cite{FSGT1990}
determined the strong chromatic index of Kneser graphs; in particular, $\chi'_s(\mathrm{O}_k)=2k-1$. The following lemma recalls the canonical
optimal strong edge-coloring of the odd graph \(\mathrm{O}_k\), described
in~\cite[Section~2]{LMSS2022}; this coloring is unique up to automorphism
of \(\mathrm{O}_k\).
\begin{lemma}[\cite{LMSS2022}]
    \label{odd graph3}
    Let $k\ge 2$. For an edge $AB$ of $\mathrm{O}_k$, where $A$ and $B$ are disjoint $(k-1)$-subsets of $[2k-1]$, we define $\sigma(AB) = [2k-1] \setminus (A \cup B)$. Then $\sigma : E(\mathrm{O}_k)\to [2k-1]$ gives a strong $(2k-1)$-edge-coloring of $\mathrm{O}_k$, and it is the unique strong $(2k-1)$-edge-coloring of $\mathrm{O}_k$ up to automorphism of $\mathrm{O}_k$.
\end{lemma}

\subsection{Density inequality}
All lower bounds rest on the next local observation.

\begin{lemma}
\label{lem:local}
Let $G$ be an edge-colored graph in which every $P_4$ is rainbow, and suppose $\delta(G) \ge 3$. Then, for every edge $xy \in E(G)$, we have all edges incident with $x$ or $y$ have pairwise distinct colors. In particular, $d_G(x) + d_G(y) - 1$ colors appear on these edges.
\end{lemma}
\begin{proof}
Suppose that two distinct edges incident with $x$ or $y$ have the same color. If they are $xy$ and $xz$, choose $w \in N(y) \setminus \{x, z\}$; then $zxyw$ is a non-rainbow $P_4$. The case where the two edges are $xy$ and $yz$ is symmetric.

If the two edges are $xz$ and $xz'$, choose $w\in N(z')\setminus \{x,z\}$; then $z x z' w$ is a non-rainbow $P_4$. The case at $y$ is analogous. Finally, let the two edges be $xz$ and $yw$. If $z\neq w$, then $z x y w$ is a non-rainbow $P_4$. If $z = w$, choose $r\in N(x)\setminus \{y,z\}$; then $r x z y$ is a non-rainbow $P_4$. This contradiction proves the lemma.
\end{proof}

\begin{corollary}\label{cor:strong-equals-rainbow}
If $G$ is a graph with $\delta(G)\ge3$, then $\chi^\prime_s(G)=\chi_{\mathrm{S}}(G,P_4)$.
\end{corollary}
\begin{proof}
By Lemma~\ref{lem:cover}, we have $\chi_{\mathrm{S}}(G,P_4)\le \chi^\prime_s(G)$. 
Conversely, let $\varphi$ be an edge-coloring of $G$ with $\chi_{\mathrm{S}}(G,P_4)$ colors in which every copy of $P_4$ is rainbow. We show that every color class of $\varphi$ is an induced matching. First, two adjacent edges cannot have the same color, since applying
Lemma \ref{lem:local} to either
of the two adjacent
edges gives a contradiction.
Second, suppose that two disjoint edges of the same color have an edge $xy$ joining their endpoints.
Then these two same-colored edges are both incident with $x$ or $y$. Applying Lemma~\ref{lem:local}
to the edge $xy$ gives a contradiction. Hence each color class is an induced matching, so $\varphi$ is a strong edge-coloring. Therefore $\chi^\prime_s(G)\le \chi_{\mathrm{S}}(G,P_4).$ Together with the first inequality, this proves the equality.
\end{proof}

Recall that for $q \in \mathbb{N}$, $f(q)$ is defined as 
\[
f(q) = \begin{cases}
\frac{d}{2} & \text{if } q = 2d-1,\\[4pt]
\frac{d(d+1)}{2d+1} & \text{if } q = 2d.
\end{cases}
\]
We shall use the fact that $f(q)$ is strictly increasing. Indeed, $\frac{d}{2} < \frac{d(d+1)}{2d+1} < \frac{d+1}{2}$ for every $d \ge 1$.

\begin{lemma}
\label{lem:density}
Let $G$ be an $n$-vertex graph with $e > 2n$ edges. If $G$ has an edge-coloring with at most $q$ colors in which every $P_4$ is rainbow, then we have 
\begin{itemize}
    \item[(i)] $q\ge \frac{4e}{n}-1$. The equality holds if and only if $G$ is a $d$-regular graph with $d\ge 3$ and $q=2d-1$. 
    \item[(ii)] $\frac{e}{n} \le f(q)$. 
\end{itemize}
\end{lemma}
\begin{proof}
Repeatedly delete vertices of current degree at most two. Let the remaining graph be $H$, with $|V(H)| = p$ and $|E(H)| = m$. If the process deleted all vertices, then it would have removed at most $2n$ edges, contrary to $e > 2n$. Thus, $H$ is nonempty. Moreover, $\delta(H)\ge 3$ and $m \ge e - 2(n-p)$. 
Consequently, we have 
\begin{equation}\label{equ2.1}
    \frac{m}{p} \ge \frac{e}{n}, \tag{2.1}
\end{equation}
because $\frac{e - 2(n-p)}{p}\ge \frac{e}{n}$ if and only if $(n-p)(e-2n) \ge 0$. 
Moreover, $\frac{m}{p} = \frac{e}{n}$ holds if and only if $n=p$ (i.e. $G=H$). 
By Lemma \ref{lem:local}, for every edge $xy \in E(H)$, $d_H(x) + d_H(y) - 1 \le q$. 
Summing over all edges of $H$, we obtain
\begin{equation}\label{equ2.2}
    \sum_{v\in V(H)} d_H(v)^2 = \sum_{xy\in E(H)} (d_H(x)+d_H(y)) \le (q+1)m. \tag{2.2}
\end{equation}

(i) The Cauchy--Schwarz inequality gives 
\[
(q+1)m\ge \sum_{v\in V(H)} d_H(v)^2 \ge \frac{1}{p}\left(\sum_{v\in V(H)} d_H(v)\right)^2=\frac{4m^{2}}{p}. 
\]
Thus we have $q\ge \frac{4m}{p}-1\ge \frac{4e}{n}-1$. 
From the discussion above, we see that $q=\frac{4e}{n}-1$ holds if and only if $G=H$ is a $d$-regular graph with $d\ge 3$ and $q=2d-1$. 

(ii) Let $q = 2d-1$. Since $(r-d)^2 \ge 0$ for every real $r$,
\[
\sum_{v\in V(H)} d_H(v)^2 \ge 2d \sum_{v\in V(H)} d_H(v) - d^2 p = 4dm - d^2 p.
\]
Together with (\ref{equ2.2}), this gives $4dm - d^2 p \le 2dm$, and hence $m/p \le d/2=f(2d-1)$.

Now let $q = 2d$. Since $(r-d)(r-d-1) \ge 0$ for every integer $r$, we have $r^2 \ge (2d+1)r - d(d+1)$. 
Applying this to the degrees in $H$ yields
\[
\sum_{v\in V(H)} d_H(v)^2 \ge 2(2d+1)m - d(d+1)p.
\]
Combining with (\ref{equ2.2}) and $q+1 = 2d+1$, we obtain $(2d+1)m \le d(d+1)p$. 
Therefore $m/p \le d(d+1)/(2d+1) = f(2d)$. Together with (\ref{equ2.1}), this completes the proof.
\end{proof}

\subsection{Strong chromatic index for regular graphs}
A graph $G$ is said to \emph{cover} a graph $H$ if there is a surjective graph homomorphism $\phi :V(G)\to V(H)$ that maps the neighbors of each vertex of $G$ bijectively onto the neighbors of its image. The surjective graph homomorphism $\phi$ is called a \emph{covering projection}.
\begin{lemma}
    \label{fiber}
    Let $G$ and $H$ be two graphs such that $G$ covers $H$. Let $\phi: V(G)\to V(H)$ be a covering projection. Then the following properties hold. 
    \begin{itemize}
        \item[(i)] For any two distinct vertices $v,v'\in V(G)$, if $\phi(v)=\phi(v')$, then $N_{G}(v)\cap N_{G}(v')=\emptyset$. 
        \item[(ii)] If $H$ is connected, then $|\phi^{-1}(x)|=|\phi^{-1}(y)|$ for any two vertices $x,y\in V(H)$. Consequently, $\lvert V(H)\rvert \mid \lvert V(G)\rvert$. 
    \end{itemize}
\end{lemma}
\begin{proof}
    (i) Suppose not, pick $w\in N_{G}(v)\cap N_{G}(v')$. Since $v,v'\in N_{G}(w)$ and $\phi(v)=\phi(v')$, $\phi$ does not map $N_{G}(w)$ bijectively onto $N_{H}(\phi(w))$, a contradiction. 
    
    (ii) Since $H$ is connected, it is sufficient to show that $|\phi^{-1}(x)|=|\phi^{-1}(y)|$ for any edge $xy\in E(H)$. By definition, for any $v\in \phi^{-1}(x)$, $\phi$ maps $N_{G}(v)$ bijectively onto $N_{H}(x)$. Since $y\in N_{H}(x)$, there is exactly one vertex $w_{v}\in N_{G}(v)$ such that $\phi(w_{v})=y$. By (i), for any two vertices $v,v'\in \phi^{-1}(x)$, $N_{G}(v)\cap N_{G}(v')=\emptyset$ holds. So we have $w_{v}\ne w_{v'}$. Therefore, we obtain $|\phi^{-1}(x)|\le|\phi^{-1}(y)|$. By symmetry, we also have $|\phi^{-1}(y)|\le|\phi^{-1}(x)|$. Hence $|\phi^{-1}(x)|=|\phi^{-1}(y)|$. 
\end{proof}

We also use the following characterization of Lu\v{z}ar, M\'{a}\v{c}ajov\'a, \v{S}koviera, and Sot\'ak~\cite[Theorem~2.2]{LMSS2022}.

\begin{theorem}[{Lu\v{z}ar, M\'{a}\v{c}ajov\'a, \v{S}koviera, and Sot\'ak~\cite[Theorem~2.2]{LMSS2022}}]
\label{thm:Luzar}
Let $k \ge 2$. A finite $k$-regular graph has strong chromatic index $2k-1$ if and only if it covers the Kneser graph $\mathrm{KG}(2k-1,k-1)$.
\end{theorem}

\subsection{Kneser examples}

In this subsection, we focus on the upper bound. In order to show $\chi_{\mathrm{S}}(n,e,P_4) \le q$, by Lemma~\ref{lem:cover}, it is sufficient to construct an $n$-vertex graph with $e$ edges and strong chromatic index $q$. The following lemma enables the construction of sufficiently large graphs with prescribed density and strong chromatic index at most $q$. 

\begin{lemma}
\label{lem:copy}
Let $F$ be a fixed graph whose strong chromatic index is $q$, and suppose $\lambda = {|E(F)|}/{|V(F)|}$. 
Let $\{e_n\}$ be any sequence of nonnegative integers with $e_n/n \to u$ as $n\to +\infty$. Then, there exists $n_0>0$ such that, if one of the following conditions holds: 
\begin{itemize}
    \item $u<\lambda$ and $n>n_0$, 
    \item $u=\lambda$, $e_{n}=un$ and $\lvert V(F)\rvert\mid n$, 
\end{itemize}
then there is an $n$-vertex graph with exactly $e_n$ edges whose strong chromatic index is at most $q$. Consequently, we have $\chi_{\mathrm{S}}(n,e_n,P_4) \le q$. 
\end{lemma}
\begin{proof}
Let $v = |V(F)|$ and $e= |E(F)| = \lambda v$. Let the graph $G_{n}$ be constructed by taking $t = \lfloor n/v\rfloor$ vertex-disjoint copies of $F$ and adding $n - tv$ isolated vertices. It is easy to see that $G_{n}$ has $te$ edges and strong chromatic index $q$. 

If $u=\lambda$, $e_{n}=un$ and $v\mid n$ for some $n$, then we have $e_{n}=un=\lambda n$ and $|E(G_{n})|=te=\lambda n$. Hence $G_{n}$ is the desired graph. 

If $u<\lambda$, since $t v > n - v$, we have $te= \lambda t v > \lambda (n - v)$. 
As $e_n/n \to u < \lambda$ and $\lambda\frac{n-v}{n} \to \lambda$, there exists $n_0$ such that $\lambda (n - v)>e_n$ holds for all $n>n_0$. Now we choose $n>n_0$ and construct $G'_{n}$ by deleting arbitrary $te-e_{n}$ edges from $G_{n}$. Since $G'_{n}$ is a subgraph of $G_{n}$, the strong chromatic index of $G'_{n}$ is at most $q$, and $e(G'_n)=te-(te-e_n)=e_n$, as desired. 
\end{proof}

We will use Lemma~\ref{lem:copy} to obtain the upper bound in Theorem~\ref{thm:coarse} and Theorem~\ref{thm:fixed}. Therefore, it suffices to construct a graph with strong chromatic index $q$ and density as large as possible.
Note that the bound in Lemma~\ref{lem:density} (ii) is also sharp. We will construct a graph $F_q$ that attains the bound in Lemma~\ref{lem:density} (ii) and apply Lemma~\ref{lem:copy} to it. 

\begin{lemma}
\label{lem:Kneser}
For every $q \ge 1$, there is a finite graph $F_q$ whose strong chromatic index is $q$ and such that ${|E(F_q)|}/{|V(F_q)|} =f(q)$. 
\end{lemma}
\begin{proof}
For $q=1$, take $F_1 = K_2$. Its single edge requires exactly one color, and $|E(F_1)|/|V(F_1)| = 1/2 = f(1)$.

Next take $q = 2d-1$, where $d \ge 2$. We have $f(q)=d/2$. Let $F_q = \mathrm{O}_d = \mathrm{KG}(2d-1,d-1)$. 
By Proposition~\ref{prop:Kneser} and Theorem~\ref{thm:Luzar}, we see that $F_q$ is the desired graph. 

Finally, take $q = 2d$, where $d \ge 1$. We start with $\mathrm{O}_{d+1} = \mathrm{KG}(2d+1,d)$, which is $(d+1)$-regular and has strong chromatic index $2d+1$. 
As in Lemma~\ref{odd graph3}, we give a canonical strong edge-coloring of $\mathrm{O}_{d+1}$ with $2d+1$ colors. By symmetry, all color classes have equal size. Delete one color class and let the remaining graph be $F_q$. Then $F_q$ has strong chromatic index at most $2d$ and has edge-density
\[
\frac{|E(F_q)|}{|V(F_q)|}=\frac{d+1}{2}\left(1 - \frac{1}{2d+1}\right) = \frac{d(d+1)}{2d+1} = f(2d).
\]

It remains to show that $\chi^\prime_s(F_q)\ge 2d$. Let $x$ be the deleted color. 
A vertex $A$ with $x\in A$ loses no incident edge, and hence has degree $d+1$ in $F_q$. 
A vertex $B$ with $x\notin B$ loses exactly one incident edge, and hence has degree $d$ in $F_q$.

Choose an edge $AB$ of $F_q$ with $x\in A$ and $x\notin B$. Such an edge exists: take any
$d$-set $A$ containing $x$ and any $d$-set $B$ disjoint from $A$; then the color of $AB$ is not $x$,
so this edge remains in $F_q$.

In any strong edge-coloring, all edges incident with $A$ or $B$ must receive pairwise distinct colors.
Therefore
\[
\chi'_s(F_q)\ge d_{F_q}(A)+d_{F_q}(B)-1=(d+1)+d-1=2d.
\]
\end{proof}

\begin{corollary}
\label{cor:upper}
Let $q \ge 1$, and define 
\[ N_q= 
\begin{cases} 
    2, & q=1,\\[2mm] 
    \binom{2p+1}{p}, & q\ge 2,\ \text{where } p=\left\lfloor \frac q2\right\rfloor. 
\end{cases} \]
Let $u$ be fixed, and let $(e_n)$ be any sequence of nonnegative integers with $e_n/n \to u$. There exists $n_0=n_0(q, (e_n))>0$ such that, if one of the following conditions holds: 
\begin{itemize}
    \item $u<f(q)$ and $n>n_0$, 
    \item $u=f(q)$, $e_n=un$ and $N_q\mid n$, 
\end{itemize}
then $\chi_{\mathrm{S}}(n,e_n,P_4) \le q$. 
\end{corollary}
\begin{proof}
Apply Lemma \ref{lem:copy} to the graph $F_q$ from Lemma \ref{lem:Kneser}.
\end{proof}

\section{Proofs }\label{sec3}

\begin{proof}[\underline{Proof of Theorem \ref{thm:coarse}}]
    For the upper bound, we compute $f(4\lfloor u \rfloor+3)=\lfloor u \rfloor+1>u$. By Corollary~\ref{cor:upper}, applied with $e_n = \lfloor u n\rfloor$, we have $\chi_{\mathrm{S}}(n,\lfloor u n\rfloor ,P_4)\le 4\lfloor u \rfloor+3$ for all $n>n_0$, where $n_0=n_0(u)$ is a constant. 

    For the lower bound, let $G$ be an $n$-vertex graph with $e$ edges admitting a $P_4$-rainbow coloring with $q$ colors. If $e > 2n$, by Lemma~\ref{lem:density} (i), we have $q\ge \frac{4e}{n}-1$. 
    If $0 < e \le 2n$, then $\frac{4e}{n} - 7 \le 1$, and the trivial bound $q \ge 1$ gives $q \ge \frac{4e}{n} - 7$. 
    Consequently $\chi_{\mathrm{S}}(n,e,P_4) \ge \frac{4e}{n} - 7$ for all $e>0$, and $\chi_{\mathrm{S}}(n,e,P_4) \ge \frac{4e}{n} - 1$ when $e>2n$. Substituting $e = \lfloor u n\rfloor$ proves the two lower bounds.
\end{proof}

\begin{proof}[\underline{Proof of Theorem \ref{thm:fixed}}]
    Let $u$ be fixed. If $0<u<1$, then $\lfloor u n\rfloor \le n-1$ for all $n>\frac{1}{1-u}$. A star with $\lfloor u n\rfloor$ edges, together with isolated vertices, contains no $P_4$, so $\chi_{\mathrm{S}}(n,\lfloor u n\rfloor ,P_4)=1$, which shows (i). 

Let $u=1$. The Erd\H os--Gallai theorem~\cite{EG1959} gives
    \[
    \mathrm{ex}(n,P_4)=
    \begin{cases}
    n, & 3\mid n,\\
    n-1, & 3\nmid n.
    \end{cases}
    \]

    If $3\mid n$, the disjoint union of $n/3$ triangles has $n$ vertices, $n$ edges, and contains no copy of $P_4$. Therefore $\chi_{\mathrm S}(n,n,P_4)=1$. 
    Suppose that $3\nmid n$. Since
    $\mathrm{ex}(n,P_4)=n-1$, every $n$-vertex graph with $n$ edges contains a copy of $P_4$. Thus
    $\chi_{\mathrm S}(n,n,P_4)\ge3$.

    For the upper bound, let $H_4$ be obtained from a triangle $abc$ by adding a pendant edge $ad$. Color $ad$ with color $1$, color $ab$ and $ac$ with color $2$, and color $bc$ with color $3$. The two copies $dabc$ and $dacb$ of $P_4$ are rainbow.
    Let $H_5$ be obtained from $H_4$ by adding a vertex $e$ adjacent to $d$, and color $de$ with color $3$. up to reversal,  the copies of $P_4$ in $H_5$ are $dabc, dacb, edab, edac,$ and each is rainbow.

    If $n=3t+1$, take the disjoint union of $H_4$ and $t-1$ triangles.
    If $n=3t+2$, take the disjoint union of $H_5$ and $t-1$ triangles.
    In either case, the resulting graph has $n$ vertices and $n$ edges and admits a $3$-coloring in which every copy of $P_4$ is rainbow. Hence $\chi_{\mathrm S}(n,n,P_4)\le3$, proving (ii).

    Suppose that $1<u\le2$. For all sufficiently large $n$, $\lfloor un\rfloor>n\ge\mathrm{ex}(n,P_4)$. Hence every $n$-vertex graph with $\lfloor un\rfloor$ edges contains a copy of $P_4$, and therefore $\chi_{\mathrm S}(n,\lfloor un\rfloor,P_4)\ge3$.
    By the definition of $q_+(u)$, we have $f(q_+(u))>u$.
    Corollary~\ref{cor:upper}, applied with $e_n=\lfloor un\rfloor$, gives $\chi_{\mathrm S}(n,\lfloor un\rfloor,P_4)\le q_+(u)$ for all sufficiently large $n$. This proves (iii).

    It remains to consider $u>2$. Since $f(q_+(u))>u$, Corollary~\ref{cor:upper} gives
$\chi_{\mathrm S}(n,\lfloor un\rfloor,P_4)\le q_+(u)$ for all sufficiently large $n$.

For the lower bound, put $q=q_0(u)-1$. Since $f$ is strictly increasing, the definition of $q_0(u)$ gives $f(q)<u$. For all sufficiently large $n$, we have
$\lfloor un\rfloor>2n$ and $\lfloor un\rfloor/n>f(q)$. Lemma~\ref{lem:density} therefore excludes a coloring with at most $q$ colors in which every copy of $P_4$ is rainbow. Thus $\chi_{\mathrm S}(n,\lfloor un\rfloor,P_4)\ge q+1=q_0(u)$.

If $2u \notin \mathbb{Z}$, then $u \neq f(q)$ for any odd integer $q$, since
$f(q)=d/2$ implies $2u=d \in \mathbb{Z}$.
For even $q=2d$, we have $\frac{d}{2} < f(2d) < \frac{d+1}{2}$. Hence, if $u = f(2d)$ for some $d$, then $d = \lfloor 2u \rfloor$, so $q = 2\lfloor 2u \rfloor$. Therefore, if $2u \notin \mathbb{Z}$ and $u \neq f(2\lfloor 2u\rfloor)$, then
$u \neq f(q)$ for any integer $q$. Hence $q_0(u)=q_+(u)$. Thus $\chi_{\mathrm S}(n,\lfloor un\rfloor,P_4)=q_+(u)$ for all sufficiently large $n$. This proves (iv).

Let $u>2$ be an integer. Since $f(4u-1)=u<f(4u)$, part (iv) gives
\begin{equation}
\label{eq:integer-critical}
4u-1\le\chi_{\mathrm S}(n,un,P_4)\le4u
\end{equation}
for all sufficiently large $n$.

Suppose first that $N_u\mid n$. Applying Corollary~\ref{cor:upper} with $q=4u-1$ gives $\chi_{\mathrm S}(n,un,P_4)\le4u-1$. Together with \eqref{eq:integer-critical}, this yields $\chi_{\mathrm S}(n,un,P_4)=4u-1$.

Now suppose that $N_u\nmid n$. Assume, for a contradiction, that there is an $n$-vertex graph $G$ with $un$ edges such that $\chi_{\mathrm S}(G,P_4)=4u-1$. Since $4u-1=4|E(G)|/|V(G)|-1$, equality holds in Lemma~\ref{lem:density}(i). Hence $G$ is $2u$-regular.

Since $\delta(G)=2u\ge3$, Corollary~\ref{cor:strong-equals-rainbow} gives $\chi'_s(G)=4u-1$. By Theorem~\ref{thm:Luzar}, the graph $G$
covers $\mathrm O_{2u}=\mathrm{KG}(4u-1,2u-1)$. By Proposition~\ref{prop:Kneser}, this graph is connected and has $N_u=\binom{4u-1}{2u-1}$ vertices. Lemma~\ref{fiber} therefore implies that $N_u\mid n$, a contradiction.

It follows from \eqref{eq:integer-critical} that $\chi_{\mathrm S}(n,un,P_4)=4u$ whenever $N_u\nmid n$. This proves (v) and completes the proof. 
\end{proof}

We use the following theorem of Alon, Moitra, and Sudakov~\cite[Theorem~2.20]{AMS2013}.

\begin{theorem}[{Alon, Moitra, and Sudakov~\cite[Theorem 2.20]{AMS2013}}]
\label{thm:Alon}
There is an absolute constant $c>0$ such that, for every $\eta>0$ and for arbitrarily large $N$, there is an $N$-vertex graph missing at most $N^{3/2+\eta}$ edges whose edge set is covered by at most $N^{2-c\eta^3}$ induced matchings.
\end{theorem}

\begin{proof}[\underline{Proof of Theorem \ref{thm:near}}]
Fix $0<\epsilon<1/2$, and choose $0<\eta < 1/2 - \epsilon$. By Theorem \ref{thm:Alon}, for arbitrarily large $n$ there exists an $n$-vertex graph $G_n$ missing at most $n^{3/2+\eta}$ edges and covered by at most $n^{2-c\eta^3}$ induced matchings. Since $3/2+\eta < 2-\epsilon$, the graph $G_n$ has more edges than required for all sufficiently large such $n$. We delete arbitrary edges until exactly $\lfloor n^{2-\epsilon}\rfloor$ edges are missing, and call the resulting graph $H_n$.

Edge deletion preserves the property of being covered by induced matchings. Lemma \ref{lem:cover} gives $\chi_{\mathrm{S}}(H_n,P_4) \le n^{2-c\eta^3}$. 
Since $|E(H_n)| = \binom{n}{2} - \lfloor n^{2-\epsilon}\rfloor$, 
we obtain $\chi_{\mathrm{S}}\left(n,\binom{n}{2} - \lfloor n^{2-\epsilon}\rfloor ,P_4\right) \le n^{2-c\eta^3}$. 
The theorem follows with $\gamma = c\eta^3$. For instance, taking $\eta = (1/2 - \epsilon)/2$ gives $\gamma = \Omega((1/2-\epsilon)^3)$.
\end{proof}

\section{One concluding remark}
Recall our answer to Problem \ref{problem-2}: for any $0<\epsilon<1/2$, there are $n$-vertex graphs with $\binom{n}{2}-\lfloor n^{2-\epsilon}\rfloor$ edges whose $P_4$-rainbow colorings use only $o(n^2)$ colors. Hence the quadratic lower bound proposed by Burr, Erd\H os, Graham, and S\'os~\cite{BEGS1989} does not extend to this range. The construction is that of Alon, Moitra, and Sudakov~\cite{AMS2013}. Whether a quadratic bound reappears for $\epsilon\ge 1/2$ is not known.


\begin{thebibliography}{99}

\bibitem{AMS2012} N. Alon, A. Moitra, and B. Sudakov, Nearly complete graphs decomposable into large induced matchings and their applications, \emph{Proceedings of the 44th Symposium on Theory of Computing Conference, STOC 2012, New York, NY, USA, May 19--22, 2012} (2012), 1079--1089.

\bibitem{AMS2013} N. Alon, A. Moitra, and B. Sudakov, Nearly complete graphs decomposable into large induced matchings and their applications, \emph{J. Eur. Math. Soc.} (2013), no. 15, 1075--1096.

\bibitem{BEFGS1991} S. A. Burr, P. Erd\H os, P. Frankl, R. L. Graham, and V. T. S\'os, Further results on maximal anti-ramsey graphs, in \emph{Graph Theory, Combinatorics, and Applications}, Vol. I, Y. Alavi, A. Schwenk (Editors), John Wiley and Sons, New York (1988), 193--206. 

\bibitem{BEGS1989} S. A. Burr, P. Erd\H os, R. L. Graham, and V. T. S\'os, Maximal antiramsey graphs and the strong chromatic number, \emph{J. Graph Theory} \textbf{13} (1989), no. 3, 263--282.

\bibitem{BCM2026} M. Buci\'c, K. Chen, and J. Ma, On a maximal anti-Ramsey conjecture of Burr, Erd\H os, Graham, and S\'os, arXiv:2603.18952, 2026.

\bibitem{E1962} P. Erd\H{o}s, On a theorem of Rademacher-Tur\'{a}n, \emph{Illinois J. Math} \textbf{6} (1962), 122--127. 

\bibitem{E1991} P. Erd\H{o}s, Problems and results in combinatorial analysis and combinatorial number theory, \emph{Graph theory, combinatorics, and applications} \textbf{1} (Kalamazoo, MI, 1988) (1991), 397--406. 


\bibitem{ESS1975} P. Erd\H os, M. Simonovits, and V. T. S\'os, Anti-Ramsey theorems, in \emph{Infinite and finite sets}, (Colloq., Keszthely, 1973; dedicated to P. Erd\H{o}s on his 60th birthday), Vol. II, North Holland, Amsterdam, Vol. 10 of Colloq Math Soc J\'{a}nos Bolyai, 633--643. 



\bibitem{EG1959} P. Erd\H os and T. Gallai, On maximal paths and circuits of graphs, \emph{Acta Math. Acad. Sci. Hungar.} \textbf{10} (1959), 337--356.

\bibitem{FG2014} Z. F\"uredi and D. S. Gunderson, Extremal numbers for odd cycles, \emph{Combin. Probab. Comput.} \textbf{24} (2015), 641--645.

\bibitem{FHS2017} J. Fox, H. Huang, and B. Sudakov, On graphs decomposable into induced matchings of linear sizes, \emph{Bull. Lond. Math. Soc.} \textbf{49} (2017), no. 1, 45--57.

\bibitem{FSGT1990} R.~J. Faudree, R.~H. Schelp, A. Gy\'{a}rf\'{a}s and Zs. Tuza, The strong chromatic index of graphs, \emph{Ars Combin.} \textbf{29} (1990), 205--211.

\bibitem{LMSS2022} B. Lu\v{z}ar, E. M\'{a}\v{c}ajov\'a, M. \v{S}koviera, and R. Sot\'ak, Strong edge colorings of graphs and the covers of Kneser graphs, \emph{J. Graph Theory} \textbf{100} (2022), no. 4, 686--697.

\bibitem{SS2006} G. N. S\'ark\"ozy, and S. M. Selkow, On an anti-Ramsey problem of Burr, Erd\H os, Graham, and T. S\'os, \emph{J. Graph Theory} \textbf{52} (2006), no. 2, 147--156.


\end{thebibliography}
\end{document}